\documentclass[11pt]{article}
\usepackage{eurosym}
\usepackage{amssymb}
\usepackage{amsfonts}
\usepackage{amsmath}
\usepackage{graphicx}
\usepackage{hyperref}
\usepackage{subcaption}
\usepackage{cancel}
\usepackage{tikz}
\usepackage{float}
\usepackage{cite}
\usepackage{float}

\newtheorem{theorem}{Theorem}

\newtheorem{lemma}[theorem]{Lemma}

\newtheorem{remark}[theorem]{Remark}

\newenvironment{proof}[1][Proof]{\textbf{#1.} }{\ \rule{0.5em}{0.5em}}
\begin{document}

\author{Pelin G. Geredeli \\
School of Mathematical and Statistical Sciences,\\ Clemson University, Clemson-SC 29634, USA \\ pgerede@clemson.edu }
\title{ An inf-sup Approach to $C_0$-Semigroup Generation for An Interactive Composite Structure-Stokes PDE Dynamics
}
\maketitle

\begin{abstract}

In this work, we investigate the existence and uniqueness properties of a composite structure (multilayered) fluid interaction PDE system which arises in multi-physics problems, and particularly in biofluidic applications related to the mammalian blood transportation process. The PDE system under consideration consists of the interactive coupling of 3D Stokes flow and 3D elastic dynamics which gives rise to an additional 2D elastic equation on the boundary interface between these 3D PDE systems. 

By means of a nonstandard mixed variational formulation we show that the PDE system generates a $C_0$-semigroup on the associated finite energy space of data. In this work, the presence of the pressure term in the 3D Stokes equation adds a great challenge to our analysis. To overcome this difficulty, we follow a methodology which is based on the necessarily non-Leray based elimination of the associated pressure term, via appropriate nonlocal operators. Moreover, while we express the fluid solution variable via decoupling of the Stokes equation, we construct the elastic solution variables by solving a mixed variational formulation via a Babuska-Brezzi approach. 

\vskip.3cm \noindent \textbf{Key terms:} Wellposedness, Fluid-Structure Interaction, Stokes Flow, Strongly Continuous Semigroups  

\vskip.3cm \noindent \textbf{Mathematics Subject Classification (MSC): 35A01, 35A02, 35A09}

\end{abstract}

\section{Introduction}

\hspace{0.6cm}Fluid structure interaction phenomena make their appearence in a variety of natural settings, e.g; cellular dynamics, blood transportation processes within arterial walls, locamotive dynamics of fish, birds and insects, design of aircraft, etc. \cite{av-trig, FSIforBIO, buk, Chambolle, Courtand, Du, hsu, KochZauzua, SunBorMulti}. Since all these dynamics are governed via multi-physics PDE models, the analysis of solutions to these PDEs is a pertinent mathematical challenge. Accordingly, a comprehensive analytic theory for such systems is mostly not available despite an extensive research activity in the last 20 years. 

With respect to the literature, coupled heat-wave PDE systems (and some variations) have been considered in the past where the heat equation component is regarded as a simplification of the fluid flow component of the fluid-structure interaction (FSI) dynamics, and the wave equation is regarded as a simplification of the structural (elastic) component; see e.g., \cite[Section 9]{lions1969quelques} and \cite{RauchZhangZuazua}. Also, the FSI dynamics in which the fluid PDE component of fluid-structure interactions is governed by Stokes or Navier-Stokes flow were studied in various settings \cite{av-trig, AvalosTriggiani09, Barbu, Chambolle, Courtand, Du, lions1969quelques}. 

We note that in the works mentioned above only “single layered” FSI models – i.e., FSI models in which only one (three dimensional) elastic PDE appears to describe the structural dynamics – have been considered. However, if one is interested in the mathematical modeling of vascular blood flow --which has been a popular subject of study of the last decade-- then the methodologies to analyze the qualitative properties of the above mentioned single-layered FSI dynamics might have limited applicability. Indeed, a given modeling PDE dynamics should  account for the fact that blood-transporting vessels are generally composed of several tissue layers, each with different constitutive properties. (See \cite{multi-layered} for more details). Moreover, mammalian blood vascular walls, being composed of viscoelastic materials, undergo large deformations due to hemodynamic forces generated during the blood transport process. As such, there is a coupling of respective blood flow and wall deformation dynamics. This physiological interaction between arterial walls and blood flow plays a crucial role in the physiology and pathophysiology of the human cardiovascular system \cite{ap-1, hsu, BorisSimplifiedFSI, SunBorMulti}, and can be mathematically realized by the appropriate FSI PDE system. In such FSI modeling, the blood flow is governed by the fluid flow PDE component (incompressible Stokes or Navier Stokes); the displacements along the elastic vascular wall are described by structural PDE components (e.g., Lam$\acute{e}$ systems of elasticity). 

As we said, it is understood that since the vascular wall structures are typically not single-layered, some degree of physical realism is lost if an arterial wall is taken to have no composite layers. Moreover, it is seen in biomedical applications that many devices (such as stents) are being developed with the view that vascular wall structures are manifestations of composite materials \cite{buk, multi-layered, SunBorMulti}. 

However, in contrast to the growing literature on single layered FSI PDE systems, mathematical analysis for the composite structure materials  is relatively limited. An initial contribution to this problem is the pioneering paper \cite{SunBorMulti} where multilayered FSI is composed of 2D (thick layer) wave equation and 1D wave equation (thin layer) coupled to a 2D fluid PDE across a boundary interface. In  \cite{SunBorMulti}, the authors proved the wellposedness of the said multilayered coupled system using a partitioned, loosely coupled scheme. This work represented a crucial result for future studies, since it showed that the presence of a thin structure with mass at the fluid–structure interface indeed regularizes the FSI dynamics.

Motivated by this work, the authors in \cite{AGM}, considered a multilayered ``canonical" 3D heat-2D wave-3D wave coupled system with the objective of analyzing the existence-uniqueness and the asymptotic behavior of the corresponding solutions. They obtained the wellposedness result using the semigroup approach, particularly, via an appropriate invocation of the Lax-Milgram Theorem. In order to obtain the asymptotic decay to the zero state, the authors of \cite{AGM} subsequently investigated the spectrum of the corresponding $C_0-$semigroup generator of the PDE system and showed that its spectrum does not intersect the imaginary axis. 

Subsequently, in \cite{AA}, a multilayered Lam$\acute{e}$-heat PDE system was considered, and an analogous asymptotic decay property was proved. Unlike the previous work \cite{AGM}, where spectral properties were necessarily obtained with respect to the generator of the corresponding PDE system, the methodology in \cite{AA} was based on the pointwise resolvent condition introduced by Y. Tomilov \cite{tomilov}, which allowed the authors to avoid the sort of technical PDE multipliers invoked in \cite{AGM}. This asymptotic stability result has recently been improved in \cite{RD}, where the authors proved that the solution to the multilayered ``canonical" 3D heat - 2D wave - 3D wave coupled system considered in \cite{AGM}, actually satisfies a rational decay rate with respect to smooth initial data.

While it is surely interesting to see that canonical multilayered structure fluid interaction PDE system manifests qualitative properties analogous to ``single layered" FSI systems, an immediate followup question is ``Can one still show the existence of a unique solution to a composite (multilayered) structure-fluid interaction PDE system, if the canonical heat component is replaced with a more realistic Stokes flow?"  With the objective of giving a positive answer to this question, in this manuscript, we consider a composite structure FSI PDE model, where the coupling of the 3D Stokes flow and 3D elastic (structural) PDE components is realized via an additional 2D elastic system on the boundary interface. We note that the boundary interface is not assumed here to evolve with time. However, it is well-accepted that for various FSI phenomena, where the boundary interface displacements between structure and fluid are small relative to the scale of the geometry, the resulting static interface models are physically relevant and reliable (see \cite{Du} and \cite{lions1969quelques}).

Our main goal in the current paper is to establish a semigroup wellposedness of the coupled Stokes-wave-Lam$\acute{e}$ PDE model. However, unlike the papers mentioned above, the presence of the pressure term gives rise to significant mathematical challenges--not seen in the aforecited works-- in determining the existence and uniqueness of solutions, and hence this term requires a nonstandard elimination for semigroup generation of the corresponding dynamical system. This constitutes the crux of the matter in our analysis, at least initially. Moreover, having the “thin” elastic dynamics on the interaction interface, as well as the matching velocity boundary conditions for fluid and structure variables prevent us from applying the standard techniques derived for single layered FSI PDEs, or uncoupled fluid flows. In this regard, \textit{to the best of the author's knowledge there are no results about the wellposedness of a composite structure-Stokes FSI system with an elastic interface. }

\section{Composite structure-Stokes fluid interaction PDE model}

\noindent In what follows, we describe the composite structure (multilayered)-Stokes fluid interaction PDE system under consideration: Let the fluid geometry $\Omega _{f}$ $\subseteq \mathbb{R}^{3}$ be a
Lipschitz, bounded domain with exterior boundary $\Gamma _{f}$. The
structure domain $\Omega _{s}$ $\subseteq \mathbb{R}^{3}$ is
\textquotedblleft completely immersed\textquotedblright\ in $\Omega _{f}$ (See figure below) 

\begin{center}
\includegraphics[scale=0.4]{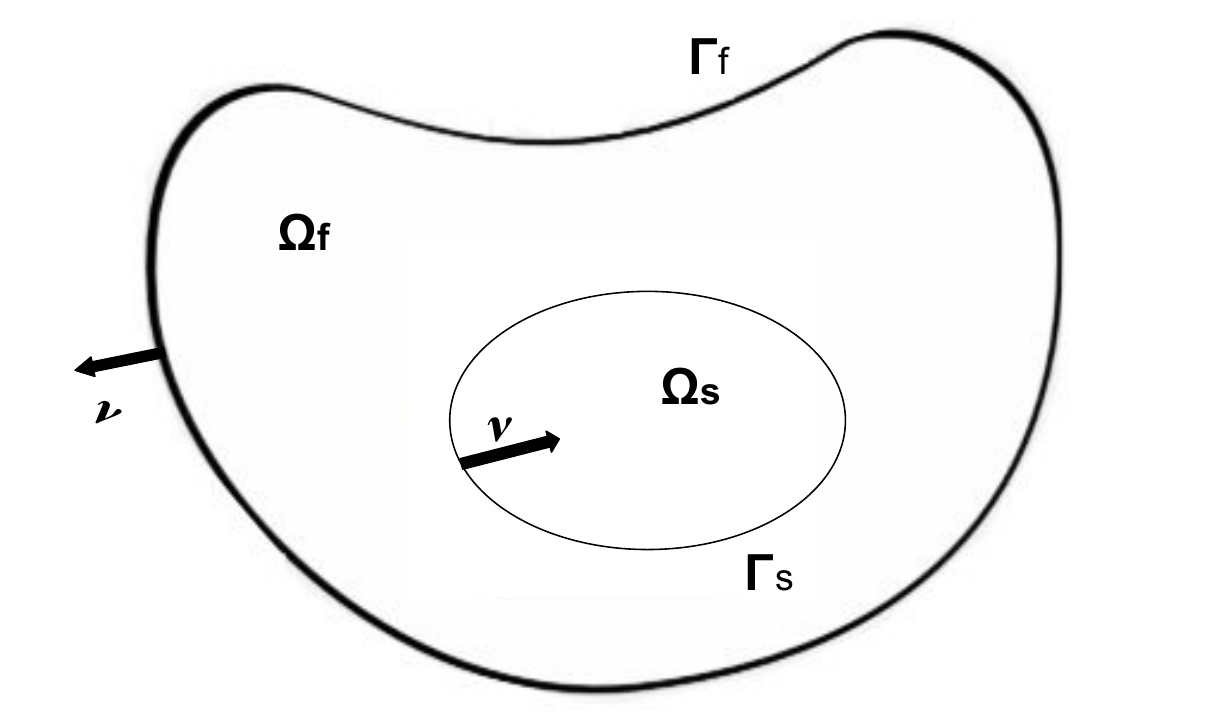}

\textbf{Figure: Geometry of the FSI Domain}
\end{center}

\noindent The boundary $\Gamma _{s}=\partial \Omega _{s},$ between $\Omega _{f}$ and $%
\Omega _{s}$ is of class $C^2,$ and given as the interaction interface between the fluid and
structure dynamics. In addition, $\nu (.)$ is the unit normal vector which is outward with
respect to $\Omega _{f},$ and so inwards with respect to $\Omega _{s}.$ With the geometry $%
\{\Omega _{s},\Omega _{f}\}$ as given, the PDE system under consideration is:

\begin{equation}
\left\{ 
\begin{array}{l}
u_{t}-\text{div}(\nabla u+\nabla ^{T}u)+\nabla p=0\text{ \ \ \ in \ }%
(0,T)\times \Omega _{f} \\ 
\text{div}(u)=0\text{ \ \ \ \ \ \ \ \ \ \ \ \ \ \ \ \ \ \ \ \ \ \ \  \ \ \ \ \ \  in \ }(0,T)\times \Omega _{f} \\ 
u|_{\Gamma _{f}}=0\text{ \ \ \ \ \ \ \ \ \ \ \ \ \ \ \ \ \ \ \ \ \ \ \  \ \ \ \ \ \  \ \ \ on \ }(0,T)\times \Gamma _{f};%
\end{array}%
\right.   \label{2a}
\end{equation}
\begin{equation}
\left\{ 
\begin{array}{l}
h_{tt}-\Delta _{\Gamma _{s}}h=[\nu \cdot
\sigma (w)]|_{\Gamma _{s}}-[\nu \cdot (\nabla u+\nabla ^{T}u)]|_{\Gamma
_{s}}+p \nu \text{ \ \ \ on \ }(0,T)\times \Gamma _{s},\text{ \ \ }%
\end{array}%
\right.   \label{2.5b}
\end{equation}%
\begin{equation}
\left\{ 
\begin{array}{l}
w_{tt}-\text{div}\sigma (w)+w=0\text{ \ \ \ \ \ \ \ \ \ \ \ \  \ on \ }(0,T)\times \Omega _{s} \\ 
w_{t}|_{\Gamma _{s}}=h_{t}=u|_{\Gamma _{s}}\text{ \ \ \ \ \ \ \ \ \ \ \ \ \ \ \ \ \ \ \ \ on \ }(0,T)\times
\Gamma _{s}%
\end{array}%
\right.   \label{2d}
\end{equation}%
\begin{equation}
\lbrack u(0),h(0),h_{t}(0),w(0),w_{t}(0)]=[u_{0},h_{0},h_{1},w_{0},w_{1}]\in 
\mathbf{H}.  \label{IC}
\end{equation}%
Here, $\Delta _{\Gamma _{s}}(.)$ is the Laplace Beltrami operator, and the stress tensor $\sigma(.) $ constitute the Lam$\acute{e}$ system of elasticity on
the ``thick" layer. Namely, for function $v$ in $\Omega _{s},$ 
\[
\sigma (v)=2\mu \epsilon (v)+\lambda \lbrack I_{3}\cdot \epsilon (v)]I_{3},
\]%
where strain tensor $\epsilon (\cdot )$ is given by 
\[
\epsilon _{ij}(v)=\frac{1}{2}\left( \frac{\partial v_{j}}{\partial x_{i}}+%
\frac{\partial v_{i}}{\partial x_{j}}\right) ,\text{ \ \ }1\leq i,j\leq 3.
\]
Also, $\mathbf{H}$ is the finite energy space defined in \eqref{H} below. 
\begin{remark}
 For the sake of numerical computation, the structure geometry $\Omega _{s}\subset 
\mathbb{R}
^{3},$ can also be taken as a convex polyhedral domain with polygonal
boundary faces $\Gamma _{j},$ $1\leq j\leq K$, where $\Gamma _{i}\cap \Gamma
_{j}\neq \emptyset $ for $i\neq j,$ and, 
\[
\Gamma _{s}=\cup _{j=1}^{K}\overline{\Gamma }_{j}.
\]%

\begin{center}
\includegraphics[scale=0.4]{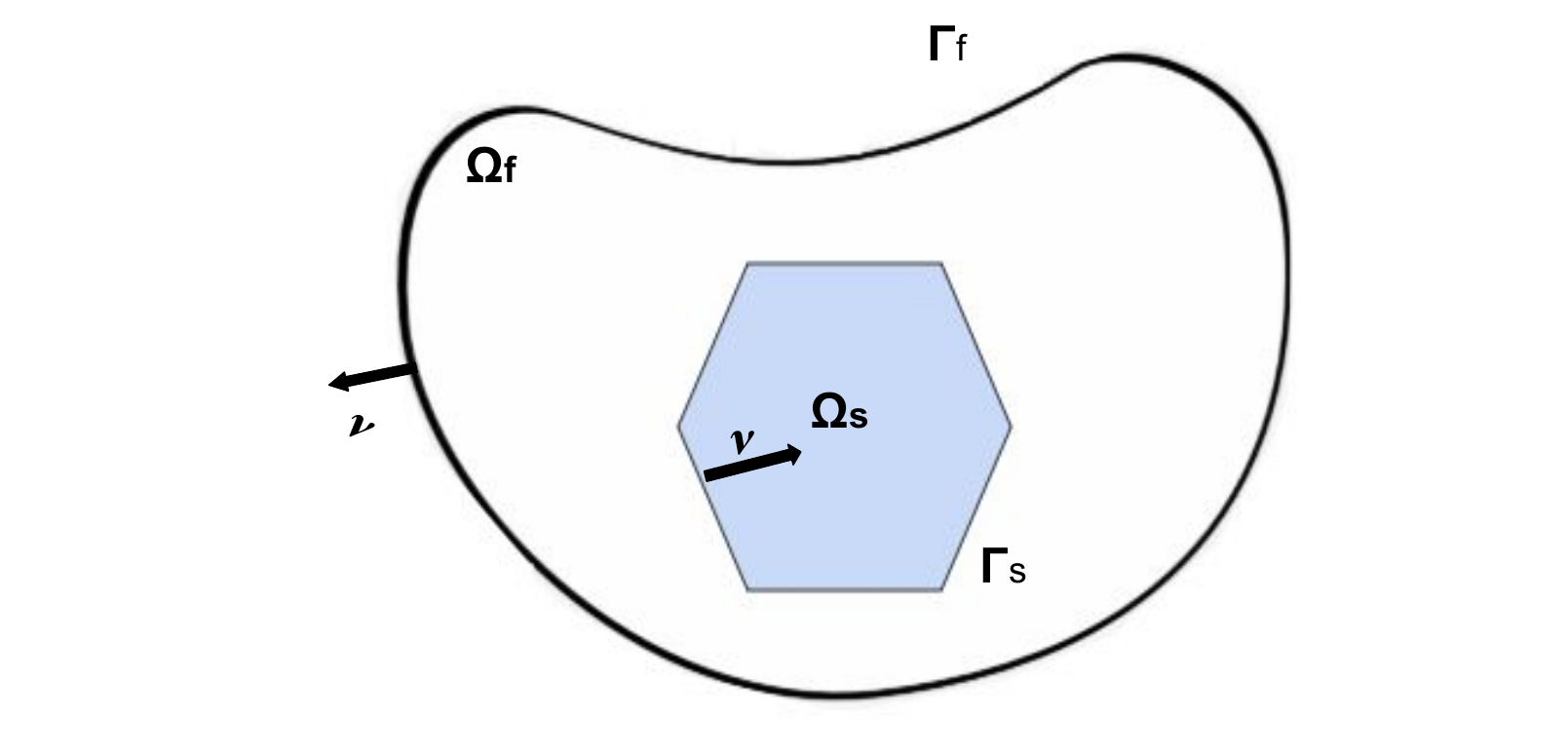}

\textbf{Figure: Alternative Geometry of the FSI Domain}
\end{center}
In this case, the thin wave equation can be modeled for $j=1,...,K$ as%
\[
\left\{ 
\begin{array}{l}
\frac{\partial ^{2}}{\partial t^{2}}h_{j}-\Delta h_{j}=[\nu \cdot
\sigma (w)]|_{\Gamma _{j}}-[\nu \cdot (\nabla u+\nabla ^{T}u)]|_{\Gamma
_{j}}+p \nu 
\text{ \ \ \ on \ }(0,T)\times \Gamma _{j} \\ 
h_{j}|_{\partial \Gamma _{j}\cap \partial \Gamma _{l}}=h_{l}|_{\partial
\Gamma _{j}\cap \partial \Gamma _{l}}\text{ on \ }(0,T)\times (\partial
\Gamma _{j}\cap \partial \Gamma _{l})\text{,\ } \forall~ 1\leq l\leq K\text{
s.t. }\partial \Gamma _{j}\cap \partial \Gamma _{l}\neq \emptyset  \\ 
\left. \dfrac{\partial h_{j}}{\partial n_{j}}\right\vert _{\partial \Gamma
_{j}\cap \partial \Gamma _{l}}=-\left. \dfrac{\partial h_{_{l}}}{\partial
n_{l}}\right\vert _{\partial \Gamma _{j}\cap \partial \Gamma _{l}}\text{on \ 
}(0,T)\times (\partial \Gamma _{j}\cap \partial \Gamma _{l})\text{,\ } \forall~ %
1\leq l\leq K\text{ s.t }\partial \Gamma _{j}\cap \partial \Gamma
_{l}\neq \emptyset .\text{\ }%
\end{array}%
\right. 
\]%
where the Laplace Beltrami Operator $\Delta _{\Gamma _{s}}(.)$ in \eqref{2.5b} is replaced with the standard Laplace
operator with the imposition of additional continuity and boundary conditions in
order to satisfy the surface differentiation (See \cite{AGM, RD}.)
   
\end{remark}
With respect to the PDE system given in (\ref{2a})-(\ref%
{IC}), the finite energy Hilbert space $\mathbf{H}$ is given as
\begin{equation}
\begin{array}{l}
\mathbf{H}=\{[u_{0},h_{0},h_{1},w_{0},w_{1}]\in [L^{2}(\Omega _{f})]^3\times
[H^{1}(\Gamma _{s})]^2\times [L^{2}(\Gamma _{s})]^2\times \\ 
\text{ \ \ \ \ \ \ \ \ \ \ \ }\times [H^{1}(\Omega _{s})]^3\times [L^{2}(\Omega
_{s})]^3: \text{ div}(u_{0})=0,\text{ \ \ }u_{0}\cdot \nu |_{\Gamma
_{f}}=0\text{, and }w_{0}|_{\Gamma _{s}}=h_{0}\text{ }\}%
\end{array}
\label{H}
\end{equation}%
with the inner product%
\begin{eqnarray}
\left\langle\Phi _{0},\widetilde{\Phi }_{0}\right\rangle_{\mathbf{H}} &=&\left\langle u_{0},\widetilde{u}%
_{0}\right\rangle_{\Omega _{f}}+\left\langle\nabla _{\Gamma _{s}}(h_{0}),\nabla _{\Gamma _{s}}(%
\widetilde{h}_{0})\right\rangle_{\Gamma _{s}}+\left\langle h_{1},\widetilde{h}_{1}\right\rangle_{\Gamma _{s}}
\nonumber \\
&&+\left\langle\sigma (w_{0}),\epsilon (\widetilde{w}_{0})\right\rangle_{\Omega _{s}}+\left\langle w_{0},%
\widetilde{w}_{0}\right\rangle_{\Omega _{s}}+\left\langle w_{1},\widetilde{w}_{1}\right\rangle_{_{\Omega _{s}}},
\label{Hilbert}
\end{eqnarray}%
where%
\begin{equation}
\Phi _{0}=\left[ u_{0},h_{0},h_{1},w_{0},w_{1}\right] \in \mathbf{H}\text{;
\ }\widetilde{\Phi }_{0}=\left[ \widetilde{u}_{0},\widetilde{h}_{0},%
\widetilde{h}_{1},\widetilde{w}_{0},\widetilde{w}_{1}\right] \in \mathbf{H}.
\label{stat}
\end{equation}

\subsection{Notation}

For the remainder of the text, norms $||\cdot ||_{D}$ are taken to be $%
L^{2}(D)$ for the domain $D$. Inner products in $L^{2}(D)$ are written $%
<\cdot ,\cdot >_{D}$, and the inner products $L_{2}(\partial D)$ are written $%
\langle \cdot ,\cdot \rangle_{\partial D} $. The space $H^{s}(D)$ will denote the Sobolev
space of order $s$, defined on a domain $D$, and $H_{0}^{s}(D)$ denotes the
closure of $C_{0}^{\infty }(D)$ in the $H^{s}(D)$ norm which we denote by $%
\Vert \cdot \Vert _{H^{s}(D)}$ or $\Vert \cdot \Vert _{s,D}$. We make use of
the standard notation for the trace of functions defined on a Lipschitz
domain $D$; i.e. for a scalar function $\phi \in H^{1}(D)$, we denote $%
\gamma (w)$ to be the trace mapping from $H^{1}(D)$ to $H^{1/2}(\partial D)$%
. We will also denote pertinent duality pairings as $(\cdot ,\cdot
)_{X\times X^{\prime }}$.

\subsection{Novelty and Technical Approach}

Qualitative properties of the composite (multilayered) version of coupled heat-wave PDE systems are currently of acute interest \cite{SunBorMulti, AGM, AA, RD}. In the benchmark composite structure-fluid interaction work \cite{SunBorMulti}, the wellposedness of a coupled  2D (thick layer) wave equation and 1D wave equation (thin layer), which interacts via a 2D fluid PDE, was established using a partitioned, loosely coupled scheme. Subsequently, by way of gaining a qualitative understanding of such FSI systems, the authors in \cite{AGM} undertook an investigation of the (higher dimensional) 3D fluid (heat)-2D (thin layer) wave-3D (thick) wave coupled PDE system. In this work, strongly continuous semigroup wellposedness of the multilayered FSI model, in the natual finite energy space, was established. Since the coupled system under study in (\ref{2a})-(\ref%
{IC}) is ``canonical",-- in particular, the fluid PDE is taken to be a heat equation,-- the strategy to show maximality of the corresponding semigroup generator required a relatively straightforward invocation of the Lax-Milgram Theorem.
\medskip

\noindent Unlike the papers mentioned above, in the present manuscript, the presence of the pressure term in the 3D fluid (Stokes) equation introduces an additional solution variable and the need to eliminate it presents a great challenge in the analysis. The essential difficulty lies in showing the maximality (range condition) of the semigroup generator, and the need to effectively address the pressure component. This FSI pressure cannot be eliminated by the Leray projector as for uncoupled flows. 

\noindent In this regard, the main challenges associated with our analysis and the novelties in this manuscript are as follows:\\

\textbf{i)} \textit{Elimination of the pressure term:}  We will establish semigroup wellposedness of the given coupled PDE system. In particular, we will associate solutions of the multilayered FSI dynamics with a $C_0$-semigroup of contractions. In this connection, our major challenge in this work becomes establishing the maximality of the candidate semigroup generator, due to the presence of the pressure variable in the fluid (Stokes) equation component of the coupled dynamics. In particular, since the boundary coupling between fluid and elastic structure components rule out the application of the Leray/Helmoltz projector, (which is appropriate for non-slip boundary conditions), we need to eliminate the pressure term via a nonstandard approach. This is actually the crux of the matter, which we address as a first step in our analysis. This approach is suggested by previous FSI control theory works, which is based on identifying the variable $p(x,t)$ as the solution to a certain elliptic boundary value problem (BVP). By critically utilizing this approach, we express the pressure variable via the other  fluid,  and ``thin" and ``thick" layer solution variables, as well as given boundary data. 

\medskip

\textbf{ii)} \textit{Formulation of nonstandard mixed variational problem and Babuska-Brezzi Approach:} Having eliminated the pressure variable from the fluid (Stokes) equation, our next step is to generate a mixed (nonstandard) variational formulation. The solvability of this saddle point problem relies on the Babuska-Brezzi (inf-sup) approach.  Indeed, this inf-sup approach will establish the maximal dissipativity (range condition) of the semigroup generator of the coupled system. At this point, we emphasize that because of the matching fluid and structure velocities, as well as the presence of “thin” elastic dynamics on the boundary interface, the argument we follow here for the necessary Babuska-Brezzi (inf-sup) formulation is very different than that given for single layered FSI, or uncoupled fluid flows. In particular, \textbf{(i)} our mixed variational formulation is created with respect to ``thin" and ``thick" layer elastic solution components, \underline{not} the Stokes flow component; \textbf{(ii)} the underlying constraint is coming from the matching velocities, and it is \textit{not} the usual divergence free constraint for incompressible flows. Subsequently, once the structure variables are solved via the mentioned mixed variational formulation, the fluid variable is constructed via the resolved ``thick" and ``thin" structural PDE components.
\medskip

\textbf{iii)} \textit{Establishment of the inf-sup condition:} Solvability of the mixed variational problem obtained from the static (resolvent) equations relies on the so called inf-sup condition of the Babuska-Brezzi Theorem, and is the key in resolving mixed variational equations. However, since our formulation is driven by the structure components, and \textit{not} the fluid, the argument followed here for the necessary inf-sup estimate is also different than that for uncoupled fluids. To establish the inf-sup estimate, we consider an elliptic BVP with the normal vector as a source term, and use the solution to the mixed variational problem to derive the estimate and the Babuska-Brezzi constant $\beta >0.$  

\section{Preliminaries and the elimination of the pressure}

\noindent The objective here is to establish the wellposedness of the system (\ref{2a})-(\ref{IC}%
) by equating it with a strongly continuous semigroup posed on the aforesaid finite energy space $\mathbf{H}.$ In this regard, the first task to
be undertaken is to eliminate the pressure variable in (\ref{2a})-(\ref{IC}%
). To this end, we note that if $p$ is a valid pressure for the composite structure FSI
system (\ref{2a})-(\ref{IC}), then after applying the divergence operator to the Stokes equation in
(\ref{2a}), and using the fact that $u$ is solenoidal we have that the
(pointwise) pressure variable $p(t)$ is harmonic; i.e.,%
\begin{equation}
\Delta p(t)=0\text{ \ \ \ in }\Omega _{f}.  \label{p-1}
\end{equation}
Now, multiplying (\ref{2a}) by $\nu |_{\Gamma _{s}}$ and using the matching velocity
condition in (\ref{2d}), we obtain the following boundary condition for the
pressure variable $p:$%
\begin{equation}
p+\frac{\partial p}{\partial \nu }=\text{div}(\nabla (u)+\nabla
^{T}(u))\cdot \nu |_{\Gamma _{s}}+[(\nabla u+\nabla ^{T}u)\cdot \nu -\Delta
_{\Gamma _{s}}(h)-\nu \cdot \sigma (w)|_{\Gamma _{s}}]\cdot \nu |_{\Gamma
_{s}}. \label{p-2}
\end{equation}%
Accordingly, the pressure variable $p(t),$ as the solution of (\ref{p-1}%
)-(\ref{p-2}), can formally be written pointwise in time as%
\[
p(t)=\mathcal{P}_{1}(u(t))+\mathcal{P}_{2}(h(t))+\mathcal{P}_{3}(w(t)) \label{9.5}
\]%
where the harmonic functions $\mathcal{P}_{1}(u(t)),$ $\mathcal{P}_{2}(h(t))$, and $\mathcal{P}_{3}(w(t))$
solve the following elliptic BVPs:%

\begin{equation}
\left\{ 
\begin{array}{l}
\Delta \mathcal{P}_{1}(u) =0\text{ \ \ \ \ \  \ \ \ \ \ \ \ \ \ \ \ \ \ \ \ \ \ \ \ \ \ \ \ \ \ \ \ \  in \ \  }\Omega _{f}   \\
\mathcal{P}_{1}(u)=([(\nabla u+\nabla ^{T}u)]\cdot
\nu )\cdot \nu |_{\Gamma _{s}}\text{\ \ \ \ \ \ on \ \  }\Gamma _{s},\\
\frac{\partial \mathcal{P}_{1}(u)}{\partial \nu } =\text{div}(\nabla
(u)+\nabla ^{T}(u))\cdot \nu |_{\Gamma _{f}}\text{\ \ \ \ \ on \ \  }\Gamma _{f},
\end{array}%
\right. \label{9.6}
\end{equation}
 
\begin{equation}
\left\{ 
\begin{array}{l}
\Delta \mathcal{P}_{2}(h) =0\text{ \ \ \ \ \ \ \ \ \ \ \ \ \ \ \ \ \ \ \  in \ \  }\Omega _{f}  \\
\mathcal{P}_{2}(h)=-\Delta _{\Gamma
_{s}}(h)\cdot \nu |_{\Gamma _{s}}\text{\ \ \ \ \ on \ \  }\Gamma _{s},\\
\frac{\partial \mathcal{P}_{2}(h)}{\partial \nu } =0\text{\ \ \ \ \ \ \ \ \ \ \ \ \ \ \ \ \ \ \ \ \ \  on \ \  }\Gamma _{f},
\end{array}%
\right. \label{9.7}
\end{equation}%
and 
\begin{equation}
\left\{ 
\begin{array}{l}
\Delta \mathcal{P}_{3}(w) =0\text{ \ \ \ \ \ \ \ \ \ \ \ \ \ \ \ \ \ \ \ \ \ \ \ \ \  in \ \  }\Omega _{f}  \\
\mathcal{P}_{3}(w)=-[\nu \cdot \sigma
(w)|_{\Gamma _{s}}]\cdot \nu |_{\Gamma _{s}}\text{\ \ \ \ \ on \ \  }\Gamma _{s}\\
\frac{\partial \mathcal{P}_{3}(w)}{\partial \nu } =0\text{\ \ \ \ \ \ \ \ \ \ \ \ \ \ \ \ \ \ \ \ \ \ \ \ \ \ \ \  on \ \  }\Gamma _{f},
\end{array}%
\right. \label{9.8}
\end{equation}%
These $\mathcal{P}_{i}$ functions, defined as the solutions to above harmonic equations, allow us
to eliminate the pressure term in the system (\ref{2a})-(\ref{IC}). As such, the
pressure-free system can now be associated with an abstract ODE in Hilbert
space $\mathbf{H}$:

\begin{equation}
\left\{ 
\begin{array}{l}
\frac{d}{dt}\Phi (t)=\mathbf{A}\Phi (t)\\
\Phi (0)=\Phi _{0}.
\label{ODE}
\end{array}%
\right. 
\end{equation}
where $\Phi (t)=\left[ u(t),h(t),h_{t}(t),w(t),w_{t}(t)\right]
,$ $\Phi _{0}=[u_{0},h_{0},h_{1},w_{0},w_{1}].$ Here, the operator $\mathbf{A}:D(\mathbf{A})\subset 
\mathbf{H}\rightarrow \mathbf{H}$ is defined by%
\begin{equation}
\mathbf{A}=\left[ 
\begin{array}{ccccc}
\text{div}(\nabla (.)+\nabla ^{T}(.)) & 0 & 0 & 0 & 0 \\ 
0 & 0 & I & 0 & 0 \\ 
-[\nu \cdot (\nabla (.)+\nabla ^{T}(.))]|_{\Gamma _{s}} & \Delta _{\Gamma
_{s}}(.) & 0 & \nu \cdot \sigma (\cdot )|_{\Gamma _{s}} & 0 \\ 
0 & 0 & 0 & 0 & I \\ 
0 & 0 & 0 & \text{div}\sigma (\cdot )-I & 0%
\end{array}%
\right]   \nonumber
\end{equation}%
\begin{equation}
+\left[ 
\begin{array}{ccccc}
\text{-}\nabla \mathcal{P}_{1}(.) & \text{-}\nabla \mathcal{P}_{2}(.) & 0 & \text{-}\nabla
\mathcal{P}_{3}(.) & 0 \\ 
0 & 0 & 0 & 0 & 0 \\ 
\mathcal{P}_{1}(.) \nu  & \mathcal{P}_{2}(.) \nu  & 0 & \mathcal{P}_{3}(.) \nu  & 0 \\ 
0 & 0 & 0 & 0 & 0 \\ 
0 & 0 & 0 & 0 & 0%
\end{array}%
\right].  \label{gener}
\end{equation}%
The domain $D(\mathbf{A})$ of $\mathbf{A}$ is characterized as follows: $\left[ u_{0},h_{0},h_{1},w_{0},w_{1}\right] \in D(\mathbf{A}) \Longleftrightarrow $ \\

\noindent \textbf{(A.i)} \text{ \ } $u_{0}\in [H^{1}(\Omega _{f})]^3,$
\text{ \ \  }$h_{1}\in [H^{1}(\Gamma _{s})]^2,$\text{ \ \  }%
$w_{1}\in [H^{1}(\Omega _{s})]^3,$

\vspace{0.3cm}

\noindent \textbf{(A.ii)} There exists an associated $L^2(\Omega_f)$-function $p_0=p_0(u_0,h_0,w_0)$ such that $$[\text{div}(\nabla u_{0}+\nabla
^{T}u_{0})-\nabla p_{0}]\in L^{2}(\Omega _{f}).$$ Consequently, $p_0$ is harmonic and so \\ \\
\textbf{(a)} $[p_{0}|_{\partial \Omega_f},\frac{\partial p_0}{\partial \nu}|_{\partial \Omega_f}]\in H^{-\frac{1}{2}}(\partial \Omega_f)\times H^{-\frac{3}{2}}(\partial \Omega_f);$\\
\textbf{(b)} $(\nabla u_{0}+\nabla
^{T}u_{0})\cdot \nu \in H^{-\frac{3}{2}}(\partial \Omega _{f}),$

 \vspace{0.3cm}

\noindent \textbf{(A.iii)} \text{div}$\sigma
(w_{0})\in L^{2}(\Omega _{s});$ \text{ \  consequently,\  } $\nu \cdot \sigma
(w_{0}) \in H^{-\frac{1}{2}}(\Gamma_s),$

 \vspace{0.3cm}
 
\noindent \textbf{(A.iv)}  $\Delta _{\Gamma _{s}}(h_{0})+[\nu
\cdot \sigma (w_{0})]|_{\Gamma _{s}}-[(\nabla u_{0}+\nabla ^{T}u_{0})\cdot
\nu ]|_{\Gamma _{s}}+[p_{0} \nu ]|_{\Gamma _{s}}\in L^{2}(\Gamma _{s}),$

\vspace{0.3cm}

\noindent \textbf{(A.v)} $u_{0}|_{\Gamma _{f}}=0,\ \
u_{0}|_{\Gamma _{s}}=h_{1}=w_{1}|_{\Gamma _{s}}$

 \vspace{0.3cm}

\noindent Note that this associated pressure function $p_0$ can be identified explicitly, via \begin{equation} p_0=\mathcal{P}_{1}(u_0)+\mathcal{P}_{2}(h_0)+\mathcal{P}_{3}(w_0), \label{11.5} \end{equation} where the $\mathcal{P}_{i}(.)s$ solve the problems given in \eqref{9.6}-\eqref{9.8}.

\section{Main Result: Wellposedness of the Stokes-Wave-Lam$\acute{e}$ PDE System}

The main result of this manuscript is to show that the system (\ref{2a})-(%
\ref{IC}), or equivalently the abstract ODE system \eqref{ODE} may be associated with a $C_0-$semigroup $\{e^{\mathbf{A}t}\}_{t\geq 0},$ where $\mathbf{A}:D(\mathbf{A})\subset 
\mathbf{H}\rightarrow \mathbf{H}$ is the matrix operator defined in \eqref{gener}. To this end, we will construct a mixed variational formulation which is necessarily predicated on the ``thick" and ``thin" structural PDE components. This is quite different than the inf-sup formulations which have been derived for uncoupled Stokes flow; see e.g; \cite{BF}. There is no choice in the matter: The proper mixed variational formulation must be driven here by the structural PDE components, although the associated bilinear form --a(.,.)-- written in Theorem \ref{BB} below-- necessarily takes into account the presence in (\ref{2a})-(%
\ref{IC}) of Stokes flow. Ultimately, we will arrive at an inf-sup system of the classical form \eqref{12.5}, for which we will apply the Babuska-Brezzi Theorem, whose statement here is recalled for the reader's convenience:
\begin{theorem}
\cite{kesevan}\label{BB} (Babuska-Brezzi) Let $\Sigma ,$ $V$ be Hilbert spaces and $a:\Sigma \times
\Sigma \rightarrow 
\mathbb{R}
,$ $b:\Sigma \times V\rightarrow 
\mathbb{R}
$ bilinear forms which are continuous. Let%
\[
Z=\left\{ \sigma \in \Sigma :b(\sigma ,v)=0,\text{ \ for every }v\in
V\right\} .
\]%
Assume that $a(\cdot ,\cdot )$ is $Z$-elliptic, i.e., there exists a
constant $\alpha >0$ such that 
\[
a(\sigma ,\sigma )\geq \alpha \left\Vert \sigma \right\Vert _{\Sigma }^{2},%
\text{ \ \ for every }\alpha \in Z.\text{ }
\]%
Assume further that there exists a constant $\beta >0$ such that%
\[
\sup_{\tau \in \Sigma }\frac{b(\tau ,v)}{\left\Vert \tau \right\Vert
_{\Sigma }}\geq \beta \left\Vert v\right\Vert _{V},\text{ \ \ for every }%
v\in V.
\]%
Then if $\kappa \in \Sigma $ and $l\in V,$ there exists a unique pair $%
(\sigma ,u)\in \Sigma \times V$ such that%
\begin{equation}
\left\{ 
\begin{array}{l}
a(\sigma ,\tau )+b(\tau ,u)=(\kappa ,\tau )\text{ \ \ \ }\forall \text{ }%
\tau \in \Sigma \\

b(\sigma ,v)=(l,v)\text{ \ \ \ \ \ \ \ \ \ \ \ \ \ \ \ \ }\forall \text{ }v\in V.
\end{array}
\right. \label{12.5}
\end{equation}
\end{theorem}
This inf-sup result will be invoked below to recover the ``thick" and ``thin" structural variables $[h_0,h_1,w_0,w_1]$ of the solution of the abstract resolvent equation \eqref{13.5} below, which is formally a frequency domain version of the time dependent system (\ref{2a})-(%
\ref{IC}). Subsequently, we will reconstruct (from $[h_1,w_1]$) the fluid and pressure variables $\{u_0,p_0\}$ of the solution to \eqref{13.5}, and moreover show that this fluid-structure interaction solution is in $D(\mathbf{A})$, where $\mathbf{A}$ is the matrix generator defined in \eqref{gener}. In this ``post-processing" work, we will require the following Lemma, the proof of which closely follows from that of Proposition 2 of \cite{AD}:  

\begin{lemma}
(Elliptic Regularity) \label{reg} Given a vector valued function $\mu \in \lbrack
H^{1}(\Omega _{f})]^{d}\cap Null($div), suppose there exists a scalar valued
function $\rho \in L^{2}(\Omega _{f})$ which satisfies%
\[
-\nabla \cdot (\nabla \mu +\nabla \mu ^{T})+\nabla \rho \in Null(\text{div}),
\]%
where the subspace $Null($div$)$ is defined as%
\[
Null(\text{div})=\left\{ f\in \lbrack L^{2}(\Omega _{f})]^{d}:\text{div}(f)=0%
\text{ \ in }\Omega _{f}\right\} .
\]%
Then, $\rho $ is harmonic ($\Delta \rho =0$ in $\Omega _{f}),$ and one has
the following additional boundary regularity for the pair $(\mu ,\rho ):$%
\[
\rho |_{\partial \Omega _{f}}\in H^{-1/2}(\partial \Omega _{f}),\text{ \ \ \ 
}\frac{\partial \rho }{\partial \nu }|_{\partial \Omega _{f}}\in
H^{-3/2}(\partial \Omega _{f});
\]%
\[
(\nabla \mu +\nabla \mu ^{T})\cdot \nu |_{\partial \Omega _{f}}\in \lbrack
H^{-1/2}(\partial \Omega _{f})]^{d}
\]%
\[
\lbrack \nabla \cdot (\nabla \mu +\nabla \mu ^{T})]\cdot \nu |_{\partial
\Omega _{f}}\in H^{-3/2}(\partial \Omega _{f}).
\]
\end{lemma}
We  now give our main result: 

\begin{theorem}
\label{well}With reference to problem (\ref{2a})-(\ref{IC}), the operator $\mathbf{A}:D(\mathbf{A})\subset \mathbf{H}%
\rightarrow \mathbf{H}$, defined in (\ref{gener}), generates a $%
C_{0}$-semigroup of contractions on $\mathbf{H}$. Consequently, the solution\\ 
$\Phi (t)=\left[ u(t),h(t),h_{t}(t),w(t),w_{t}(t)\right] $ of (\ref{2a})-(\ref{IC}), or
equivalently (\ref{ODE}), is given by 
\[
\Phi (t)=e^{\mathbf{A}t}\Phi _{0}\in C([0,T];\mathbf{H})\text{,}
\]%
where $\Phi _{0}=\left[ u_{0},h_{0},h_{1},w_{0},w_{1}\right] \in \mathbf{H}$.
\end{theorem}

\begin{proof}
With an obstensibly simple view of invoking the Lumer-Phillips Theorem, we will show that the generator $\mathbf{A}$ is a maximal dissipative operator. We will give our proof in two steps:\\

\noindent \textbf{STEP 1: (Dissipativity)} \\

\noindent Given the
matrix generator $\mathbf{A},$ $\Phi _{0}=\left[
u_{0},h_{0},h_{1},w_{0},w_{1}\right] \in D(\mathbf{A}),$ and $p_0$ as in \eqref{11.5}, 
\[
\left\langle \mathbf{A}\Phi ,\Phi \right\rangle _{\mathbf{H}}=\left\langle 
\text{div}(\nabla (u_{0})+\nabla ^{T}(u_{0})),u_{0}\right\rangle _{\Omega
_{f}}+\left\langle \nabla _{\Gamma _{s}}(h_{1}),\nabla _{\Gamma
_{s}}(h_{0})\right\rangle _{\Gamma _{s}}
\]%
\[
+\left\langle -(\nabla u_{0}+\nabla ^{T}u_{0})\cdot \nu |_{_{\Gamma
_{s}}},h_{1}\right\rangle _{\Gamma _{s}}+\left\langle \Delta _{\Gamma
_{s}}(h_{0}),h_{1}\right\rangle _{\Gamma _{s}}+\left\langle \sigma
(w_{0})\cdot \nu |_{\Gamma _{s}},h_{1}\right\rangle _{\Gamma _{s}}
\]%
\[
+\left\langle \sigma (w_{1}),\epsilon (w_{0})\right\rangle _{\Omega
_{s}}+\left\langle w_{1},w_{0}\right\rangle _{\Omega _{s}}+\left\langle 
\text{div}\sigma (w_{0}),w_{1}\right\rangle _{\Omega _{s}}-\left\langle
w_{0},w_{1}\right\rangle _{\Omega _{s}}
\]%
\[
-\left\langle [\nabla \mathcal{P}_{1}(u_{0})+\nabla \mathcal{P}_{2}(h_{0})+\nabla
\mathcal{P}_{3}(w_{0})],u_{0}\right\rangle _{\Omega _{f}}
\]%
\[
+\left\langle [\mathcal{P}_{1}(u_{0})\cdot \nu +\mathcal{P}_{2}(h_{0})\cdot \nu
+\mathcal{P}_{3}(w_{0})\cdot \nu ],h_{1}\right\rangle _{\Gamma _{s}}.
\]%
Applying Green's theorem, using the fact that $u_{0}$ is solenoidal, and $u_{0}=0
$ on $\Gamma _{f},$ we get 
\[
\left\langle \mathbf{A}\Phi ,\Phi \right\rangle _{\mathbf{H}}=\left\langle
(\nabla (u_{0})+\nabla ^{T}(u_{0}))\cdot \nu ,u_{0}\right\rangle _{\Gamma
_{s}}-\frac{1}{2}\left\Vert \nabla (u_{0})+\nabla ^{T}(u_{0})\right\Vert ^{2}
\]%
\[
+\left\langle \nabla _{\Gamma _{s}}(h_{1}),\nabla _{\Gamma
_{s}}(h_{0})\right\rangle _{\Gamma _{s}}-\left\langle (\nabla (u_{0})+\nabla
^{T}(u_{0}))\cdot \nu ,h_{1}\right\rangle _{\Gamma _{s}}
\]%
\[
-\left\langle \nabla _{\Gamma _{s}}(h_{0}),\nabla _{\Gamma
_{s}}(h_{1})\right\rangle _{\Gamma _{s}}+\left\langle \sigma (w_{0})\cdot
\nu |_{\Gamma _{s}},h_{1}\right\rangle _{\Gamma _{s}}
\]%
\[
+\left\langle \sigma (w_{1}),\epsilon (w_{0})\right\rangle _{\Omega
_{s}}+\left\langle w_{1},w_{0}\right\rangle _{\Omega _{s}}-\left\langle
\sigma (w_{0})\cdot \nu |_{\Gamma _{s}},w_{1}\right\rangle _{\Gamma _{s}}
\]%
\[
-\left\langle \sigma (w_{0}),\epsilon (w_{1})\right\rangle _{\Omega
_{s}}-\left\langle w_{0},w_{1}\right\rangle _{\Omega _{s}}.
\]%
Now, using the matching velocity condition given in \textbf{(A.v)} we have that%
\[
\left\langle \mathbf{A}\Phi ,\Phi \right\rangle _{\mathbf{H}}=-\frac{1}{2}%
\left\Vert \nabla (u_{0})+\nabla ^{T}(u_{0})\right\Vert ^{2}+2i\text{Im\{}%
\left\langle \nabla _{\Gamma _{s}}(h_{1}),\nabla _{\Gamma
_{s}}(h_{0})\right\rangle _{\Gamma _{s}}+\left\langle \sigma
(w_{1}),\epsilon (w_{0})\right\rangle _{\Omega _{s}}\]%
\[
+\left\langle
w_{1},w_{0}\right\rangle _{\Omega _{s}}
+\left\langle \sigma (w_{1}),\epsilon (w_{0})\right\rangle _{\Omega
_{s}}+\left\langle w_{1},w_{0}\right\rangle _{\Omega _{s}}\text{\}}
\]%
and hence 
\[
\text{Re}\left\langle \mathbf{A}\Phi ,\Phi \right\rangle _{\mathbf{H}}=-%
\frac{1}{2}\left\Vert \nabla (u_{0})+\nabla ^{T}(u_{0})\right\Vert ^{2}\leq 0,
\]%
which gives us the dissipativity of the operator $\mathbf{A.}$\\

\noindent\textbf{STEP 2: (Maximal Dissipativity)}\\

\noindent This step is the challenging part of the proof. In order to prove the maximality condition, we show that the operator $(\lambda I-\mathbf{A})$ is surjective for $\lambda >0.$ That is, we establish the range condition%
\begin{equation}
Range(\lambda I-\mathbf{A})=\mathbf{H},  \label{range}
\end{equation}%
where parameter $\lambda >0.$ Let $\Phi ^{\ast }=\left[ u_{0}^{\ast
},h_{0}^{\ast },h_{1}^{\ast },w_{0}^{\ast },w_{1}^{\ast }\right] \in \mathbf{%
H,}$ and consider the problem of finding $\Phi =\left[
u_{0},h_{0},h_{1},w_{0},w_{1}\right] \in D(\mathbf{A})$ which solves%
\begin{equation}
(\lambda I-\mathbf{A})\Phi =\Phi ^{\ast },\text{ \ \ }\lambda >0 \label{13.5}.
\end{equation}%
In PDE terms, this resolvent equation will generate the following relations, where again $p_0$ is given via \eqref{11.5}: 
\begin{equation}
\left\{ 
\begin{array}{l}
\lambda u_{0}-\text{div}(\nabla u_{0}+\nabla ^{T}u_{0})+\nabla p_0=u_{0}^{\ast
}\text{ \ \ \ in \ }\Omega _{f} \\ 
\text{div}(u_{0})=0\text{ \ \ \ \ \ \ in \ }\Omega _{f} \\ 
u_{0}|_{\Gamma _{f}}=0\text{ \ \ \ on \ }\Gamma _{f};%
\end{array}%
\right.   \label{s1}
\end{equation}

\begin{equation}
\left\{ 
\begin{array}{c}
\begin{array}{l}
\lambda h_{0}-h_{1}=h_{0}^{\ast }\text{ \ \ in }\Gamma _{s}%
\end{array}
\\ 
\lambda h_{1}+[\nu \cdot (\nabla u_{0}+\nabla ^{T}u_{0})]|_{\Gamma
_{s}}-\Delta _{\Gamma _{s}}(h_{0})-[\nu \cdot \sigma (w_0)]|_{\Gamma
_{s}}-p_0 \nu =h_{1}^{\ast }\text{ \ \ \ in\ \ }\Gamma _{s},\text{ \ \ }%
\end{array}%
\right.   \label{s2}
\end{equation}%
\begin{equation}
\left\{ 
\begin{array}{l}
\lambda w_{0}-w_{1}=w_{0}^{\ast }\text{ \ \ in \ \ \ }\Omega _{s} \\ 
\lambda w_{1}-\text{div}\sigma (w_{0})+w_{0}=w_{1}^{\ast }\text{ \ \ \ in \ }%
\Omega _{s} \\ 
w_{1}|_{\Gamma _{s}}=h_{1}=u_{0}|_{\Gamma _{s}}\text{ \ \ \ on \ }\Gamma _{s}%
\end{array}%
\right.   \label{s3}
\end{equation}%
Because of the composite structure of the coupled dynamics, and the matching fluid and structure velocities given in (\ref{s3})$_{3},$ the
static problem (\ref{s1})-(\ref{s3}) cannot be solved via mixed variational approaches which are given for
uncoupled fluid flows (see pg. 15 and pg. 107 of \cite{temam}). Hence, our proof is based on solving an appropriate
and nonstandard mixed variational problem formulated for the static PDE system (%
\ref{s1})-(\ref{s3}). We note that our mixed variational formulation is generated and solved with respect to
the ``thin" and ``thick" structure variables $h_{1}$ and $w_{1}.$ For this, we mainly appeal to
the Babuska-Brezzi approach (see Theorem \ref{BB}). Once these
variables are solved then the recovery of the
other structure solution variables $h_{0}$ and $w_{0}$ will be via the
relations given in (\ref{s2})$_1$ and (\ref{s3})$_1$ for given data $h_{0}^{\ast }\in [H^{1}(\Gamma _{s})]^2 $ and $w_{0}^{\ast }\in [H^{1}(\Omega _{s})]^3.$  i.e., 
\begin{equation}
h_{0}=\frac{1}{\lambda }h_{1}+\frac{1}{\lambda }h_{0}^{\ast }  \label{h-0}
\end{equation}%
\begin{equation}
w_{0}=\frac{1}{\lambda }w_{1}+\frac{1}{\lambda }w_{0}^{\ast }.  \label{w-0}
\end{equation}%
We also note that, because of the matching velocities boundary conditions in \eqref{s3}, $$\int_{\Gamma_s} w_1|_{\Gamma_s} d\Gamma_s=\int_{\Gamma_s} h_1 d\Gamma_s=0.$$ Before starting to generate a weak formulation for the ``thin" and ``thick" structure variables $h_{1}$ and $w_{1},$ we first proceed with the fluid component $u_{0}$: For given $g\in [H^{1/2}(\Gamma _{s})]^2,$ the unique pair  $%
[u_{1}(g),p_{1}(g)]\in [H^{1}(\Omega _{f})]^3\times \hat{L}^{2}(\Omega _{f})%
$ solves the following static Stokes
equation:
\begin{equation}
\left\{ 
\begin{array}{l}
\lambda u_{1}-\text{div}(\nabla u_{1}+\nabla ^{T}u_{1})+\nabla p_{1}=0\text{%
\ \ \ in \ }\Omega _{f} \\ 
\text{div}(u_{1})=\frac{\int\limits_{\Gamma _{s}}(g\cdot \nu )d\Gamma _{s}%
\text{ }}{meas(\Omega _{f})}\text{\ \ \ \ \ \ in \ }\Omega _{f} \\ 
u_{1}|_{\Gamma _{s}}=g\text{ \ \ \ on \ }\Gamma _{s} \\ 
u_{1}|_{\Gamma _{f}}=0\text{ \ \ \ on \ }\Gamma _{f},%
\end{array}%
\right.   \label{decomp-1}
\end{equation}%
(See \cite[pg 22, Theorem 2.4]{temam}). Here, $$ \hat{L}^{2}(\Omega _{f})=\{f\in L^2(\Omega_f): \int_{\Omega_f} fd\Omega_f=0\}  $$In particular, the compatibility condition
holds for the solvability of the above problem, and so $p_{1}$ is
unique up to a constant (see e.g.,\cite{temam}, pg 31, Theorem
2.4). In a similar way, for a given force term $u_{0}^{\ast }\in
[L^{2}(\Omega_f)]^3,$ the unique pair $[u_{2}(u_{0}^{\ast
}),p_{2}(u_{0}^{\ast })]\in [H^{1}(\Omega _{f})]^3\times \hat{L}^{2}(\Omega _{f})$ solves  the following problem:
\begin{equation}
\left\{ 
\begin{array}{l}
\lambda u_{2}-\text{div}(\nabla u_{2}+\nabla ^{T}u_{2})+\nabla
p_{2}=u_{0}^{\ast }\text{\ \ \ in \ }\Omega _{f} \\ 
\text{div}(u_{2})=0\text{\ \ \ \ in \ }\Omega _{f} \\ 
u_{2}|_{\partial \Omega _{f}}=0\text{ \ \ \ on \ }\partial \Omega _{f}.%
\end{array}%
\right.   \label{decomp-2}
\end{equation}

\noindent With the solution maps of \eqref{decomp-1}-\eqref{decomp-2}, the unique $\{u_{0},p_0\}$ of \eqref{s1} may then be
expressed as%
\begin{equation}
u_{0}=u_{1}(w_{1}|_{\Gamma _{s}})+u_{2}(u_{0}^{\ast });\text{ \ \ \ \ \ \ \
\ }p_0=p_{1}(w_{1}|_{\Gamma _{s}})+p_{2}(u_{0}^{\ast })+c_{0},  \label{u-repr}
\end{equation}%
where $c_{0}$ is the (presently) unknown constant component of the pressure $p_0$ of (\ref%
{s1}). Define the space%
\[
\textbf{S}=\left\{ (\varphi ,\psi )\in [H^{1}(\Gamma _{s})]^2\times [H^{1}(\Omega
_{s})]^3:\varphi =\psi |_{\Gamma _{s}}\right\}.
\]%
In order to generate a mixed variational formulation for the static ``thin"
and ``thick" solution variables in (\ref{s2})-(\ref{s3}), we respectively multiply (\ref{s2})$_{2}$ and (\ref{s3})$_{2}$ by functions $\varphi \in [H^{1}(\Gamma _{s})]^2,$ and $\psi \in
[H^{1}(\Omega _{s})]^3$ of the space $\textbf{S}$, use Green's Theorem, and add the subsequent relations, (taking into account (\ref{h-0})-(\ref{w-0})) to get:%
\[
\lambda \left\langle h_{1},\varphi \right\rangle _{\Gamma _{s}}+\frac{1}{%
\lambda }\left\langle \nabla _{\Gamma _{s}}h_{1},\nabla _{\Gamma
_{s}}\varphi \right\rangle _{\Gamma _{s}}+\lambda \left\langle w_{1},\psi
\right\rangle _{\Omega _{s}}
\]%
\[
+\frac{1}{\lambda }\left\langle \sigma (w_{1}),\epsilon (\psi )\right\rangle
_{\Omega _{s}}+\frac{1}{\lambda }\left\langle w_{1},\psi \right\rangle
_{\Omega _{s}}-c_{0}\left\langle \nu ,\varphi \right\rangle _{\Gamma _{s}}
\]%
\begin{equation}
+\left\langle \nu \cdot (\nabla u_{0}+\nabla ^{T}u_{0})|_{\Gamma
_{s}},\varphi \right\rangle _{\Gamma _{s}}-\left\langle (p_{1}+p_{2})
\nu ,\varphi \right\rangle _{\Gamma _{s}}  \label{1a}
\end{equation}%
\[
=-\frac{1}{\lambda }\left\langle \nabla _{\Gamma _{s}}h_{0}^{\ast },\nabla
_{\Gamma _{s}}\varphi \right\rangle _{\Gamma _{s}}-\frac{1}{\lambda }%
\left\langle \sigma (w_{0}^{\ast }),\epsilon (\psi )\right\rangle _{\Omega
_{s}}
\]%
\begin{equation}
+\left\langle h_{1}^{\ast },\varphi \right\rangle _{\Gamma
_{s}}+\left\langle w_{1}^{\ast },\psi \right\rangle _{\Omega _{s}}-\frac{1}{%
\lambda }\left\langle w_{0}^{\ast },\psi \right\rangle _{\Omega _{s}}. 
\label{1}
\end{equation}%
In order to estimate the sum in (\ref{1a}), we recall that for any $\varphi
\in [H^{\frac{1}{2}}(\Gamma _{s})]^2,$ there is a unique pair $[\widetilde{u}(\varphi ),%
\widetilde{p}(\varphi )]\in [H^{1}(\Omega _{f})]^3\times \hat{L}^{2}(\Omega _{f})$ which solves the BVP%
\begin{equation}
\left\{ 
\begin{array}{l}
\lambda \widetilde{u}-\text{div}(\nabla \widetilde{u}+\nabla ^{T}\widetilde{u%
})+\nabla \widetilde{p}=0\text{\ \ \ in \ }\Omega _{f} \\ 
\text{div}(\widetilde{u})=\frac{\int\limits_{\Gamma _{s}}(\varphi \cdot \nu
)d\Gamma _{s}\text{ }}{meas(\Omega _{f})}\text{\ \ \ \ \ \ in \ }\Omega _{f}
\\ 
\widetilde{u}|_{\Gamma _{s}}=\varphi \text{ \ \ \ on \ }\Gamma _{s} \\ 
\widetilde{u}|_{\Gamma _{f}}=0\text{ \ \ \ on \ }\Gamma _{f}.%
\end{array}%
\right.   \label{2aa}
\end{equation}
Appealing to problem \eqref{2aa}, we have for the sum in (\ref{1a}) that 
\[
\left\langle \nu \cdot (\nabla u_{0}+\nabla ^{T}u_{0})|_{\Gamma
_{s}},\varphi \right\rangle _{\Gamma _{s}}-\left\langle (p_{1}+p_{2})\cdot
\nu ,\varphi \right\rangle _{\Gamma _{s}}
\]%
\[
=\left\langle \nabla u_{0}+\nabla ^{T}u_{0},\nabla \widetilde{u}(\varphi
)+\nabla ^{T}\widetilde{u}(\varphi )\right\rangle _{\Omega
_{f}}+\left\langle \text{div}(\nabla u_{0}+\nabla ^{T}u_{0}),\widetilde{u}%
(\varphi )\right\rangle _{\Omega _{f}}
\]%
\[
-\left\langle \nabla (p_{1}+p_{2}),\widetilde{u}(\varphi )\right\rangle
_{\Omega _{f}}-\left\langle p_{1}+p_{2},\text{div}\widetilde{u}(\varphi
)\right\rangle _{\Omega _{f}}
\]%
\[
=\left\langle \nabla u_{0}+\nabla ^{T}u_{0},\nabla \widetilde{u}(\varphi
)+\nabla ^{T}\widetilde{u}(\varphi )\right\rangle _{\Omega _{f}}+\lambda
\left\langle u_{0},\widetilde{u}(\varphi )\right\rangle _{\Omega _{f}}
\]%
\begin{equation}
-\left\langle u_{0}^{\ast },\widetilde{u}(\varphi )\right\rangle _{\Omega
_{f}}-\left\langle p_{1}+p_{2},\text{div}\widetilde{u}(\varphi
)\right\rangle _{\Omega _{f}},  \label{2.5}
\end{equation}%
where we have also used (\ref{s1})$_{1}.$ Then from the last relation, and the
fluid representation in (\ref{u-repr}) we get 
\[
\left\langle \nu \cdot (\nabla u_{0}+\nabla ^{T}u_{0})|_{\Gamma
_{s}},\varphi \right\rangle _{\Gamma _{s}}-\left\langle (p_{1}+p_{2})
\nu ,\varphi \right\rangle _{\Gamma _{s}}
\]%
\[
=\left\langle \nabla u_{1}(w_{1}|_{\Gamma _{s}})+\nabla
^{T}u_{1}(w_{1}|_{\Gamma _{s}}),\nabla \widetilde{u}(\varphi )+\nabla ^{T}%
\widetilde{u}(\varphi )\right\rangle _{\Omega _{f}}
\]%
\[
+\left\langle \nabla u_{2}(u_{0}^{\ast })+\nabla ^{T}u_{2}(u_{0}^{\ast
}),\nabla \widetilde{u}(\varphi )+\nabla ^{T}\widetilde{u}(\varphi
)\right\rangle _{\Omega _{f}}
\]%
\[
+\lambda \left\langle u_{1}(w_{1}|_{\Gamma _{s}}),\widetilde{u}(\varphi
)\right\rangle _{\Omega _{f}}+\lambda \left\langle u_{2}(u_{0}^{\ast }),%
\widetilde{u}(\varphi )\right\rangle _{\Omega _{f}}
\]%
\begin{equation*}
\boxed{-\left\langle p_{1}(w_{1}|_{\Gamma _{s}}),\text{div}\widetilde{u}(\varphi
)\right\rangle _{\Omega _{f}}-\left\langle p_{2}(u_{0}^{\ast }),\text{div}%
\widetilde{u}(\varphi )\right\rangle _{\Omega _{f}}}
\end{equation*}%
\begin{equation}
-\left\langle u_{0}^{\ast
},\widetilde{u}(\varphi )\right\rangle _{\Omega _{f}}.  \label{3}
\end{equation}
Since $p_1(w_1|_{\Gamma_s})$ and $p_2(u^*_0)$ are each in $\hat{L}^{2}(\Omega _{f}),$ then the boxed term of \eqref{3} disappears. Combining \eqref{1}-\eqref{3}, we get
\[
\lambda \left\langle h_{1},\varphi \right\rangle _{\Gamma _{s}}+\frac{1}{%
\lambda }\left\langle \nabla _{\Gamma _{s}}h_{1},\nabla _{\Gamma
_{s}}\varphi \right\rangle _{\Gamma _{s}}+\lambda \left\langle w_{1},\psi
\right\rangle _{\Omega _{s}}+\frac{1}{\lambda }\left\langle \sigma (w_{1}),\epsilon (\psi )\right\rangle
_{\Omega _{s}}
\]
\[
+\frac{1}{\lambda }\left\langle w_{1},\psi \right\rangle
_{\Omega _{s}}-c_{0}\left\langle \nu ,\varphi \right\rangle _{\Gamma _{s}}+\lambda \left\langle u_{1}(w_{1}|_{\Gamma _{s}}),\widetilde{u}(\varphi
)\right\rangle _{\Omega _{f}}
\]
\[
+\left\langle \nabla u_{1}(w_{1}|_{\Gamma _{s}})+\nabla
^{T}u_{1}(w_{1}|_{\Gamma _{s}}),\nabla \widetilde{u}(\varphi )+\nabla ^{T}%
\widetilde{u}(\varphi )\right\rangle _{\Omega _{f}}\]%
\[
=-\left\langle \nabla u_{2}(u_{0}^{\ast })+\nabla ^{T}u_{2}(u_{0}^{\ast
}),\nabla \widetilde{u}(\varphi )+\nabla ^{T}\widetilde{u}(\varphi
)\right\rangle _{\Omega _{f}}
\]
\[
-\lambda \left\langle u_{2}(u_{0}^{\ast }),%
\widetilde{u}(\varphi )\right\rangle _{\Omega _{f}}+\left\langle u_{0}^{\ast
},\widetilde{u}(\varphi )\right\rangle _{\Omega _{f}}
\]
\[
-\frac{1}{\lambda }\left\langle \nabla _{\Gamma _{s}}h_{0}^{\ast },\nabla
_{\Gamma _{s}}\varphi \right\rangle _{\Gamma _{s}}-\frac{1}{\lambda }%
\left\langle \sigma (w_{0}^{\ast }),\epsilon (\psi )\right\rangle _{\Omega
_{s}}
\]%
\begin{equation}
+\left\langle h_{1}^{\ast },\varphi \right\rangle _{\Gamma
_{s}}+\left\langle w_{1}^{\ast },\psi \right\rangle _{\Omega _{s}}-\frac{1}{%
\lambda }\left\langle w_{0}^{\ast },\psi \right\rangle _{\Omega _{s}}. 
\label{27.5}
\end{equation}%
The last relation now gives us the following mixed variational
formulation in terms of the ``thin" and ``thick" structure variables $h_{1}$ and $%
w_{1}:$ Namely,%
\begin{eqnarray}
\mathbf{a}([h_{1},w_{1}],[\varphi ,\psi ])+\mathbf{b}([\varphi ,\psi ],c_{0}) &=&\mathbf{F}([\varphi
,\psi ]),\text{ \ \ \ }\forall \text{ }[\varphi ,\psi ]\in \textbf{S} \nonumber \\
\mathbf{b}([h_{1},w_{1}],r) &=&0,\text{ \ \ \ \ \ \ \ \ \ \ \ \ \ \ }\forall \text{ }r\in 
\mathbb{R}.
\label{vari}
\end{eqnarray}%
Here, the bilinear forms $\mathbf{a}(.,.):\textbf{S}\times \textbf{S}\rightarrow 
\mathbb{R}
$ and $\mathbf{b}(.,.):\textbf{S}\times 
\mathbb{R}
\rightarrow 
\mathbb{R}
$ are respectively given as%
\begin{eqnarray*}
\mathbf{a}([\phi ,\xi ],[\widetilde{\phi },\widetilde{\xi }]) &=&\lambda \left\langle
\phi ,\widetilde{\phi }\right\rangle _{\Gamma _{s}}+\frac{1}{\lambda }%
\left\langle \nabla _{\Gamma _{s}}\phi ,\nabla _{\Gamma _{s}}\widetilde{\phi 
}\right\rangle _{\Gamma _{s}} \\
&&+\lambda \left\langle \xi ,\widetilde{\xi }\right\rangle _{\Omega _{s}}+%
\frac{1}{\lambda }\left\langle \sigma (\xi ),\epsilon (\widetilde{\xi }%
)\right\rangle _{\Omega _{s}}+\frac{1}{\lambda }\left\langle \xi ,\widetilde{%
\xi }\right\rangle _{\Omega _{s}} \\
&&+\left\langle \nabla u_{1}(\xi |_{\Gamma _{s}})+\nabla ^{T}u_{1}(\xi
|_{\Gamma _{s}}),\nabla \widetilde{u}(\widetilde{\phi })+\nabla ^{T}%
\widetilde{u}(\widetilde{\phi })\right\rangle _{\Omega _{f}} \\
&&+\lambda \left\langle u_{1}(\xi |_{\Gamma _{s}}),\widetilde{u}(\widetilde{%
\phi })\right\rangle _{\Omega _{f}},\end{eqnarray*}%
and%
\[
\mathbf{b}([\widetilde{\phi },\widetilde{\xi }],r)=-r\left\langle \nu ,\widetilde{%
\phi }\right\rangle _{\Gamma _{s}},\]%
and the functional $\mathbf{F}(.)$ is defined as  
\begin{eqnarray*}
\mathbf{F}([\widetilde{\phi },\widetilde{\xi }]) &=&-\left\langle \nabla
u_{2}(u_{0}^{\ast })+\nabla ^{T}u_{2}(u_{0}^{\ast }),\nabla \widetilde{u}(%
\widetilde{\phi })+\nabla ^{T}\widetilde{u}(\widetilde{\phi })\right\rangle
_{\Omega _{f}} \\
&&-\frac{1}{\lambda }\left\langle \nabla _{\Gamma _{s}}h_{0}^{\ast },\nabla
_{\Gamma _{s}}\widetilde{\phi }\right\rangle _{\Gamma _{s}}-\frac{1}{\lambda 
}\left\langle \sigma (w_{0}^{\ast }),\epsilon (\widetilde{\xi }%
)\right\rangle _{\Omega _{s}} \\
&&-\lambda \left\langle u_{2}(u_{0}^{\ast }),\widetilde{u}(\widetilde{\phi }%
)\right\rangle _{\Omega _{f}}+\left\langle u_{0}^{\ast },\widetilde{u}(%
\widetilde{\phi })\right\rangle _{\Omega _{f}} \\
&&+\left\langle h_{1}^{\ast },\widetilde{\phi }\right\rangle _{\Gamma
_{s}}+\left\langle w_{1}^{\ast },\widetilde{\xi }\right\rangle _{\Omega
_{s}}-\frac{1}{\lambda }\left\langle w_{0}^{\ast },\widetilde{\xi }%
\right\rangle _{\Omega _{s}}.
\end{eqnarray*}%
The remainder of the proof hinges on properly applying Theorem \ref{BB}. It is clear that the bilinear forms $\mathbf{a}(.,.)$ and $\mathbf{b}(.,.)$ are continuous, and moreover $\mathbf{a}(.,.)$ is $\textbf{S}$-elliptic. In order to conclude that the variational problem
(\ref{vari}) has a unique solution, we need to show that the bilinear form $%
\mathbf{b}(.,.)$ satisfies the ``inf-sup" condition given in Theorem \ref{BB}. For this, we consider the following
problem: 

\noindent Given $r\in 
\mathbb{R}
$, let $z\in [H^{1}(\Gamma _{s})]^2$  satisfy%
\[
\Delta _{\Gamma _{s}}z=sgn(r)\nu \text{ \ \ \ on \ }\Gamma _{s}
\]%
It is easily seen that $\left\Vert \nabla _{\Gamma _{s}}z\right\Vert _{\Gamma
_{s}}\leq C\left\Vert \nu \right\Vert _{\Gamma _{s}}.$ Now, taking into account that $\gamma
: H^1(\Omega_s)\rightarrow H^{1/2}(\Gamma_s)$ is a surjective map, and so it has a continuous right inverse $\gamma^+(z),$ we have%
\begin{eqnarray*}
\sup_{\lbrack \eta ,\varsigma ]\in \textbf{S}}\frac{b([\eta ,\varsigma ],r)}{%
\left\Vert [\eta ,\varsigma ]\right\Vert _{\textbf{S}}} &\geq &\frac{b([z,\gamma^+(z)],r)}{%
\left\Vert z\right\Vert _{[H^{1}(\Gamma _{s})]^2}} \\
&=&\frac{-r\int\limits_{\Gamma _{s}}\nu \cdot zd\Gamma _{s}}{\left\Vert
z\right\Vert _{[H^{1}(\Gamma _{s})]^2}} \\
&=&-rsgn(r)\frac{\int\limits_{\Gamma _{s}}\Delta _{\Gamma _{s}}z\cdot
zd\Gamma _{s}}{\left\Vert z\right\Vert _{[H^{1}(\Gamma _{s})]^2}} \\
&=&\left\vert r\right\vert \frac{\int\limits_{\Gamma _{s}}|\nabla _{\Gamma
_{s}}z|^{2}d\Gamma _{s}}{\left\Vert z\right\Vert _{[H^{1}(\Gamma _{s})]^2}} \\
&=&\left\vert r\right\vert \left\Vert z\right\Vert _{[H^{1}(\Gamma _{s})]^2},
\end{eqnarray*}%
which yields that the inf-sup condition holds with the constant $\beta =\left\Vert z\right\Vert
_{[H^{1}(\Gamma _{s})]^2}.$ Consequently, the existence and uniqueness of the solution $%
[h_{1},w_{1}]\in \textbf{S}$ and $c_{0}\in 
\mathbb{R}
$  to the mixed variational problem (\ref{vari}) follows from Theorem \ref{BB}. \\ 

\noindent Now, the unique attainment of the solution components $[h_{1},w_{1}]\in \textbf{S}$ allows the subsequent recovery of the solution variables $h_0$ and $w_0$ via the resolvent relations (\ref{h-0})-(\ref{w-0}), with $$h_{0}=w_{0}|_{\Gamma _{s}},$$ since  the data $\Phi
^{\ast }=\left[ u_{0}^{\ast },h_{0}^{\ast },h_{1}^{\ast },w_{0}^{\ast
},w_{1}^{\ast }\right] \in \mathbf{H}$. Moreover, having the ``thick" structure component $$w_{1}\in [H^{1}(\Omega _{s})]^3~~ \text{with}~~ w_{1}|_{\Gamma _{s}}=h_{1}\in [H^{1}(\Gamma _{s})]^2,$$ and $u_{0}^{\ast
}\in [L^{2}(\Omega _{f})]^3,$ the fluid solution component $u_{0}$ and the pressure term $p_0$ can be
given via the expressions in (\ref{u-repr}). That is, the unique pair $\{u_0,p_0\}\in [H^1(\Omega_f)]^3\times L^2(\Omega_f)$ solves the Stokes system 
\begin{equation}
\left\{ 
\begin{array}{l}
\lambda u_{0}-\text{div}(\nabla u_{0}+\nabla ^{T}u_{0})+\nabla
p_0=u_{0}^{\ast }\text{\ \ \ in \ }\Omega _{f} \\ 
\text{div}(u_{0})=0\text{\ \ \ \ in \ }\Omega _{f} \\ 
u_{0}|_{\partial \Omega _{f}}=0\text{ \ \ \ on \ }\partial \Omega _{f}\\
u_{0}|_{\partial \Gamma _{s}}=w_1\text{ \ \ \ on \ }\Gamma _{s}.
\end{array}
\right.  \label{u-p}
\end{equation}
To conclude the proof of maximality, we must show that the derived solution $\left[
u_{0},h_{0},h_{1},w_{0},w_{1}\right]\in D(\mathbf{A})$ and satisfies the resolvent equations \eqref{s1}-\eqref{s3}. First, by \eqref{u-repr}, it is clear that $\{u_0,p_0\}\in [H^1(\Omega_f)]^3\times L^2(\Omega_f)$ where $u_0=u_1+u_2;$ $p_0=p_1+p_2+c_0$ satisfies \eqref{s1}. Consequently,\\ by Lemma \ref{reg},
$$
p_0 |_{\partial \Omega _{f}}\in H^{-1/2}(\partial \Omega _{f}),\text{ \ \ \ 
}\frac{\partial p_0 }{\partial \nu }|_{\partial \Omega _{f}}\in
H^{-3/2}(\partial \Omega _{f}),$$
and
$$(\nabla u_0 +\nabla u^{T}_0 )\cdot \nu |_{\partial \Omega _{f}}\in \lbrack
H^{-1/2}(\partial \Omega _{f})]^{d};~~~\lbrack \nabla \cdot (\nabla u_0 +\nabla^{T} u_0 )]\cdot \nu |_{\partial
\Omega _{f}}\in H^{-3/2}(\partial \Omega _{f})
$$
Moreover, if we take in \eqref{vari}, $\varphi =0,$ and $\psi\in [D(\Omega_s)]^3,$ we then have 
$$\lambda \left\langle w_{1},\psi
\right\rangle _{\Omega _{s}}+\frac{1}{\lambda }\left\langle \sigma (w_{1}),\epsilon (\psi )\right\rangle
_{\Omega _{s}}+\frac{1}{\lambda }\left\langle w_{1},\psi \right\rangle
_{\Omega _{s}}$$ 
$$=-\frac{1}{\lambda }%
\left\langle \sigma (w_{0}^{\ast }),\epsilon (\psi )\right\rangle _{\Omega
_{s}}+\left\langle w_{1}^{\ast },\psi \right\rangle _{\Omega _{s}}-\frac{1}{%
\lambda }\left\langle w_{0}^{\ast },\psi \right\rangle _{\Omega _{s}}.$$
Thus, via the resolvent relation $w_0=\frac{1}{\lambda}[w_1+w^*_0],$ we have 
\begin{equation}
\lambda w_1-\text{div} \sigma(w_0)+w_0=w^*_1~~in~~ [L^2(\Omega_s)]^3, \label{30.4} 
\end{equation}
which gives us \eqref{s3}. Since also $w_0\in [H^1(\Omega_s)]^3,$ an energy method yields in turn,
\begin{equation}
\nu \cdot \sigma(w_0)\in H^{-1/2}(\Gamma_s) \label{30.5}
\end{equation}
via \eqref{30.4}. This gives, for all $[\varphi,\psi]\in \mathbf{S},$ (after using the resolvent relation $h_0=\frac{1}{\lambda}(h_1+h^*_0)$):
\begin{equation*}
- \lambda\left\langle w_{1},\psi \right\rangle
_{\Omega _{s}}+\left\langle\text{div}\sigma (w_{0}),\psi \right\rangle
_{\Omega _{s}}-\left\langle w_{0},\psi
\right\rangle _{\Omega _{s}}=-\left\langle w_{1}^{\ast },\psi \right\rangle
_{\Omega _{s}},
\end{equation*}%
and
\[
\lambda \left\langle h_{1},\varphi \right\rangle _{\Gamma _{s}}+\left\langle \nabla _{\Gamma _{s}}h_{0},\nabla _{\Gamma
_{s}}\varphi \right\rangle _{\Gamma _{s}}+\lambda \left\langle w_{1},\psi
\right\rangle _{\Omega _{s}}+\left\langle \sigma (w_{0}),\epsilon (\psi )\right\rangle
_{\Omega _{s}}
\]%
\[
+\left\langle w_{0},\psi \right\rangle
_{\Omega _{s}}-c_{0}\left\langle \nu ,\varphi \right\rangle _{\Gamma _{s}}-\left\langle (p_{1}+p_{2})
\nu ,\varphi \right\rangle _{\Gamma _{s}}+\left\langle (p_{1}+p_{2})
\nu ,\varphi \right\rangle _{\Gamma _{s}}  
\]
\[
+\left\langle \nabla u_{0}+\nabla ^{T}u_{0},\nabla \widetilde{u}(\varphi
)+\nabla ^{T}\widetilde{u}(\varphi )\right\rangle _{\Omega _{f}}+\lambda
\left\langle u_{0},\widetilde{u}(\varphi )\right\rangle _{\Omega _{f}}
\]
\begin{equation*}
=\left\langle h_{1}^{\ast },\varphi \right\rangle _{\Gamma
_{s}}+\left\langle w_{1}^{\ast },\psi \right\rangle _{\Omega _{s}}+\left\langle u^*_{0},\widetilde{u}(\varphi )\right\rangle _{\Omega _{f}}.
\end{equation*}%
Integrating by part (having in hand the boundary trace in \eqref{30.5}), and adding the two relations now gives 
\[
\lambda \left\langle h_{1},\varphi \right\rangle _{\Gamma _{s}}+\left\langle \nabla _{\Gamma _{s}}h_{0},\nabla _{\Gamma
_{s}}\varphi \right\rangle _{\Gamma _{s}}-\left\langle \nu \cdot \sigma(w_0),\varphi \right\rangle _{\Gamma _{s}}
\]%
\[
-\left\langle p\nu,\varphi  \right\rangle+\left\langle \nabla p\nu,\varphi  \right\rangle
_{\Omega _{f}}+\lambda
\left\langle u_{0},\widetilde{u}(\varphi )\right\rangle _{\Omega _{f}}
\]
\[
-\left\langle \text{div}[\nabla u_{0}+\nabla ^{T}u_{0}], \widetilde{u}(\varphi
)\right\rangle _{\Omega _{f}}+\left\langle \nu \cdot [\nabla u_{0}+\nabla ^{T}u_{0}], \varphi
\right\rangle _{\Omega _{s}}
\]
\[
=\left\langle h_{1}^{\ast },\varphi \right\rangle _{\Gamma
_{s}}+\left\langle u^*_{0},\widetilde{u}(\varphi )\right\rangle _{\Omega _{f}} ~~~\forall~~ [\varphi,\psi]\in \mathbf{S}. 
\]
This finally gives the inference that
\[
\lambda \left\langle h_{1},\varphi \right\rangle _{\Gamma _{s}}+\left\langle \nabla _{\Gamma _{s}}h_{0},\nabla _{\Gamma
_{s}}\varphi \right\rangle _{\Gamma _{s}}-\left\langle \nu \cdot \sigma(w_0),\varphi \right\rangle _{\Gamma _{s}}
\]%
\[
-\left\langle p\nu,\varphi  \right\rangle+\left\langle \nu \cdot [\nabla u_{0}+\nabla ^{T}u_{0}], \varphi
\right\rangle _{\Omega _{s}}
\]
\[
=\left\langle h_{1}^{\ast },\varphi \right\rangle _{\Gamma
_{s}},~~\forall~~ \varphi \in [\mathcal{D}(\Omega_s)]^3,
\]
which gives the ``thin" structural equation in \eqref{s2}, and the proof of Theorem \ref{well} now finishes.
\end{proof}





\section{Acknowledgement}

The author Pelin G. Geredeli would like to thank the National Science Foundation, and acknowledge her partial funding from NSF Grant DMS-2348312.

\end{document}